\documentclass[leqno,12pt]{article} %leqno is the option to put formula numbers on the left side
\usepackage{amsmath, amssymb}
\usepackage{amsthm} %theorem environment option
\usepackage{pictexwd,dcpic}
\usepackage{xfrac}
\usepackage{xypic}

\theoremstyle{plain} %text of this environment is typesetted in italics
\newtheorem{theorem}{\indent\sc Theorem}[section]
\newtheorem{lemma}[theorem]{\indent\sc Lemma}
\newtheorem{corollary}[theorem]{\indent\sc Corollary}
\newtheorem{proposition}[theorem]{\indent\sc Proposition}

\theoremstyle{definition} %text of this environment is typesetted in roman letters

\newtheorem{remark}[theorem]{\indent\sc Remark}

\newtheorem{notation}[theorem]{\indent\sc Notation}

\numberwithin{equation}{section}

\def\F{\mathcal{F}}

\def\o{\omega}

\def\G{\Gamma}
\def\O{\mathcal{O}}
\def\Z{\zeta}

\makeatletter
\def\address#1#2{\begingroup
\noindent\parbox[t]{7.8cm}{%
\small{\scshape\ignorespaces#1}\par\vskip1ex
\noindent\small{\itshape E-mail address}%
\/: #2\par\vskip4ex}\hfill%
\endgroup}%

\makeatother
\title{{Generators of the ring of weakly holomorphic modular functions for $\G_1(N)$ }} %title of the paper
\author{
\textsc{Ja Kyung Koo and Dong Sung Yoon} %names of authors
}

\date{} %leave empty
\begin{document}

\maketitle

%%%%%%%%%%%%%%% footnote %%%%%%%%%%%%%%%%
\footnote{ %2010 MSC numbers
 }
\footnote{ %key words and phrases
} \footnote{
\thanks{
The second named author was supported by the National Institute for Mathematical Sciences, Republic of Korea.
} }
%%%%%%%%%%%%%%%%%%%%%%%%%%%%%%%%%%%%%%%%%$$$$$$$$$$$$$$$$$$$$$$$$$$$$$$$$$$$$$$$$$$$$$$$$$$$$$$$$$$$$$$$$$$$$$$$$$$$$$$$$$$$$

\begin{abstract}
For a positive integer $N$ divisible by $4,5,6,7$ or $9$, let $\mathcal{O}_{1,N}(\mathbb{Q})$ be the ring of weakly holomorphic modular functions for the congruence subgroup $\Gamma_1(N)$ with rational Fourier coefficients. 
We present explicit generators of the ring $\mathcal{O}_{1,N}(\mathbb{Q})$ over $\mathbb{Q}$ by making use of modular units which have infinite product expansions.

\end{abstract}

\tableofcontents

\maketitle

\section{Introduction}
Let $\G$ be a congruence subgroup of $\mathrm{SL}_2(\mathbb{Z})$.
A meromorphic modular function for $\G$ is called a \textit{weakly holomorphic} if its poles (if any) are supported only at the cusps of $\G$. 
For a positive integer $N$, let
\begin{equation*}
\begin{array}{ccl}
\G(N)&=&\left\{\left[
\begin{matrix}
a&b\\
c&d
\end{matrix}\right] \in \mathrm{SL}_{2}(\mathbb{Z})~\big|~
\left[\begin{matrix}
a&b\\
c&d
\end{matrix}\right]
\equiv 
\left[\begin{matrix}
1&0\\
0&1
\end{matrix}\right]
\pmod{N} \right\},\vspace{0.2cm}\\
\G_1(N)&=&\left\{\left[
\begin{matrix}
a&b\\
c&d
\end{matrix}\right] \in \mathrm{SL}_{2}(\mathbb{Z})~\big|~
\left[\begin{matrix}
a&b\\
c&d
\end{matrix}\right]
\equiv 
\left[\begin{matrix}
1&*\\
0&1
\end{matrix}\right]
\pmod{N} \right\},\vspace{0.2cm}\\
\G^1(N)&=&\left\{\left[
\begin{matrix}
a&b\\
c&d
\end{matrix}\right] \in \mathrm{SL}_{2}(\mathbb{Z})~\big|~
\left[\begin{matrix}
a&b\\
c&d
\end{matrix}\right]
\equiv 
\left[\begin{matrix}
1&0\\
*&1
\end{matrix}\right]
\pmod{N} \right\}.
\end{array}
\end{equation*}
By definition, we have $\G\supset\G(N)$ for some positive integer $N$. 
Furthermore, one can check that
\begin{equation*}
\left[
\begin{matrix}
N & 0\\
0 & 1
\end{matrix}\right]^{-1}
\G(N)
\left[
\begin{matrix}
N & 0\\
0 & 1
\end{matrix}\right]
\supset \G_1(N^2).
\end{equation*}
Hence if $f(\tau)$ is a modular function for $\G(N)$, then $f(N\tau)$ is a modular function for $\G_1(N^2)$.
Here we note that the Fourier coefficients of $f(N\tau)$ coincide with those of $f(\tau)$.
Therefore, the study of modular functions with respect to congruence subgroups is reduced to that of modular functions for $\G_1(N)$.

\par
For a rational vector $\mathbf{r}=\left[\begin{matrix}r_1\\r_2\end{matrix}\right]\in\mathbb{Q}^2\setminus\mathbb{Z}^2$, we define the \textit{Siegel function} $g_{\mathbf{r}}(\tau)$ on the complex upper half plane $\mathbb{H}$ by the following infinite product expansion
\begin{equation}\label{Siegel expansion}
g_{\mathbf{r}}(\tau)=-q^{\frac{1}{2}\mathbf{B}_2(r_1)}e^{\pi i r_2(r_1-1)}(1-q^{r_1}e^{2\pi i r_2})\prod_{n=1}^\infty(1-q^{n+r_1}e^{2\pi i r_2})(1-q^{n-r_1}e^{-2\pi i r_2}),
\end{equation}
where $\mathbf{B}_2(X)=X^2-X+1/6$ is the second Bernoulli polynomial and $q=e^{2\pi i\tau}$.
As is well known (\cite{Siegel} or \cite[p.36]{Kubert}), it is a modular unit, that is, both zeros and poles are supported at the cusps of certain congruence subgroup.
\par
Now, let $\O_{1,N}(\mathbb{Q})$ (resp. $\O^1_N(\mathbb{Q})$) be the ring of weakly holomorphic modular functions for $\Gamma_1(N)$ (resp. $\G^1(N)$) with rational Fourier coefficients.
When $N\equiv 0\pmod{4}$, Eum et al (\cite{E-K-S}) recently constructed explicit generators of $\O^1_N(\mathbb{Q})$ by means of Siegel functions and classify all Fricke families of such level $N$, which will be defined in $\S$\ref{Fricke families of level $N$}.
\par
In this paper we shall investigate how to generate the ring $\O_{1,N}(\mathbb{Q})$ for arbitrary $N$ by improving their idea.
We first construct generators of $\O_{1,N}(\mathbb{Q})$ over $\mathbb{Q}$ for $1\leq N\leq 10$ and $N=12$ by utilizing Siegel functions (Theorem \ref{generator}).
And, for given positive integers $m>3$ and $N$ divisible by $m$ we will further present a primitive generator of $\O_{1,N}(\mathbb{Q})$ over $\O_{1,m}(\mathbb{Q})$, which is a Weierstrass unit (Theorem \ref{generator2}), from which we are able to determine generators of the ring $\O_{1,N}(\mathbb{Q})$ over $\mathbb{Q}$ when $N$ is divisible by $4,5,6,7$ or $9$ (Corollary \ref{main corollary}).
Lastly, as byproduct, we can classify all Fricke families of such level $N$ (Theorem \ref{classify Fricke families}). 

\begin{notation}
The transpose of a matrix $\alpha$ is denoted by ${^t}\alpha$.
If $R$ is a ring with unity, $R^\times$ stands for the group of all invertible elements of $R$.
For a positive integer $N$, let $\zeta_N=e^{2\pi i/N}$ be a primitive $N$-th root of unity.
\end{notation}

\section{Galois actions of modular function fields}
In this section we shall describe the action of Galois groups between modular function fields.

For a positive integer $N$, let
$\mathcal{F}_N$ be the field of meromorphic modular functions for $\G(N)$ whose Fourier coefficients lie in the $N$th cyclotomic field $\mathbb{Q}(\zeta_N)$.
\begin{proposition}\label{modular function action}
$\mathcal{F}_N$ is a Galois extension of $\mathcal{F}_1$ and 
\begin{equation*}
\mathrm{Gal}(\mathcal{F}_N/\mathcal{F}_1)\cong \mathrm{GL}_2(\mathbb{Z}/N\mathbb{Z})/\{\pm I_2\}\cong G_N\cdot \mathrm{SL}_2(\mathbb{Z}/N\mathbb{Z})/\{\pm I_2\},
\end{equation*}
where 
\begin{equation*}
G_N=\left\{\left[
\begin{matrix}
1&0\\
0&d
\end{matrix}\right]~|~d\in(\mathbb{Z}/N\mathbb{Z})^\times
\right\}.
\end{equation*}
More precisely, $\mathrm{GL}_2(\mathbb{Z}/N\mathbb{Z})/\{\pm I_2\}$ acts on $\mathcal{F}_N$ as follows:

\begin{itemize}
\item[\textup{(i)}]
The element 
$\left[
\begin{matrix}
1&0\\
0&d
\end{matrix}\right]\in G_N$ acts on $h\in\mathcal{F}_N$ by
\begin{equation*}
h(\tau)=\sum_{n\gg-\infty}c_n q^{{n}/{N}}\longmapsto \sum_{n\gg-\infty}c_n^{\sigma_d} q^{{n}/{N}},
\end{equation*}
where $\sum_{n\gg-\infty}c_n q^{{n}/{N}}$ is the Fourier expansion of $h$ and 
$\sigma_d\in\mathrm{Gal}(\mathbb{Q}(\zeta_N)/\mathbb{Q})$ satisfies $\zeta_N^{\sigma_d}=\zeta_N^d$.
\item[\textup{(ii)}]
The element $\gamma\in \mathrm{SL}_2(\mathbb{Z}/N\mathbb{Z})/\{\pm I_2\}$ acts on $h\in\mathcal{F}_N$ by 
\begin{equation*}
h(\tau)^\gamma=(h\circ\widetilde{\gamma})(\tau), 
\end{equation*}
where $\widetilde{\gamma}$ is a preimage of $\gamma$ of the reduction map $\mathrm{SL}_2(\mathbb{Z})\rightarrow \mathrm{SL}_2(\mathbb{Z}/N\mathbb{Z})/\{\pm I_2\}$.
\end{itemize}
\end{proposition}

\begin{proof}
\cite[Chapter 6, Theorem 3]{Lang}.
\end{proof}

In a similar fashion as Proposition \ref{modular function action} (i), an element $\sigma\in\mathrm{Gal}(\mathbb{Q}_\mathrm{ab}/\mathbb{Q})$ induces an element of $\mathrm{Gal}(\F_N/\F_1)$ by applying $\sigma$ to the Fourier coefficients.
For $h\in\F_N$, we denote by $h^\sigma$ the image of $h$ under this automorphism.

\section{Fricke functions and Siegel functions} 

For a lattice $L$ in $\mathbb{C}$, we let
\begin{eqnarray*}
g_2(L)=60\sum_{\o\in L\setminus\{0\}}\frac{1}{\o^4},\quad 
g_3(L)=140\sum_{\o\in L\setminus\{0\}}\frac{1}{\o^6},\quad
\Delta(L)=g_2(L)^3-27g_3(L)^2.
\end{eqnarray*}
And, we define the \textit{Weierstrass $\wp$-function} (relative to $L$) and the \textit{$j$-invariant} by
\begin{eqnarray*}
\wp(z;L)&=&\frac{1}{z^2}+\sum_{\o\in L\setminus\{0\}}\left(\frac{1}{(z-\o)^2}-\frac{1}{\o^2} \right)\quad(z\in\mathbb{C}),\\
j(\tau)&=&1728\frac{g_2(\tau)}{\Delta(\tau)}\quad (\tau\in\mathbb{H}),
\end{eqnarray*}
where $g_2(\tau)=g_2([\tau,1])$, $g_3(\tau)=g_3([\tau,1])$ and $\Delta(\tau)=\Delta([\tau,1])$.
For a rational vector $\mathbf{r}=\left[\begin{matrix}r_1\\r_2\end{matrix}\right]\in\mathbb{Q}^2\setminus\mathbb{Z}^2$, the \textit{Fricke function} is defined by
\begin{equation*}
f_{\mathbf{r}}(\tau)=-2^7 3^5\frac{g_2(\tau)g_3(\tau)}{\Delta(\tau)}\wp(r_1\tau+r_2;[\tau,1])\qquad(\tau\in\mathbb{H}).
\end{equation*}
Then $\mathbb{Q}[j(\tau)]$ is the ring of weakly holomorphic functions in $\mathcal{F}_1$ and we have
\begin{equation*}
\mathcal{F}_N=\left\{
\begin{array}{ll}
\mathbb{Q}(j(\tau))&\textrm{if $N=1$}\vspace{0.1cm}\\
\mathbb{Q}\left(j(\tau),f_{\mathbf{r}}(\tau) : \mathbf{r}\in\frac{1}{N}\mathbb{Z}^2\setminus\mathbb{Z}^2\right)&\textrm{if $N>1$}
\end{array}\right.
\end{equation*}
(\cite[Chapter 5, Theorem 2 and Chapter 6, \S3]{Lang}).

\begin{proposition}\label{properties of Fricke}
Let $N~(>1)$ be an integer and $\mathbf{r}\in(1/N)\mathbb{Z}^2\setminus\mathbb{Z}^2$.
\begin{itemize}
\item[\textup{(i)}] $f_\mathbf{r}(\tau)$ is weakly holomorphic and depends only on $\pm\mathbf{r}\pmod{\mathbb{Z}^2}$.
\item[\textup{(ii)}] For $\alpha\in\mathrm{GL}_2(\mathbb{Z}/N\mathbb{Z})/\{\pm I_2\}$ we obtain
\begin{equation*}
f_\mathbf{r}(\tau)^\alpha=f_{{^t}\alpha\mathbf{r}}(\tau).
\end{equation*}
\end{itemize}
\end{proposition}

\begin{proof}
\cite[Chapter 6, \S2-\S3]{Lang}
\end{proof}

The following two propositions describe the modularity criterion for Siegel functions and the relation between Fricke functions and Siegel functions.

\begin{proposition}\label{property of Siegel functions}
Let $N~(>1)$ be an integer and $\{m(\mathbf{r})\}_{\mathbf{r}=\left[\begin{smallmatrix}r_1\\r_2\end{smallmatrix}\right]\in\frac{1}{N}\mathbb{Z}^2\setminus\mathbb{Z}^2}$ be a family of integers such that $m(\mathbf{r})=0$ except for finitely many $\mathbf{r}$.
\begin{itemize}

\item[\textup{(i)}]
If $\sum_\mathbf{r}m(\mathbf{r})\equiv 0\pmod{12}$, then for $\gamma\in\mathrm{SL}_2(\mathbb{Z})$ we get
\begin{equation*}
\prod_{\mathbf{r}}g_\mathbf{r}(\tau)^{m(\mathbf{r})}\circ \gamma=\prod_\mathbf{r}g_{{^t}\gamma \mathbf{r}}(\tau)^{m(\mathbf{r})}.
\end{equation*}
\item[\textup{(ii)}] A finite product of Siegel functions
\begin{equation*}
\zeta\prod_{\mathbf{r}}g_{\mathbf{r}}(\tau)^{m(\mathbf{r})}
\end{equation*}
belongs to $\mathcal{F}_N$ if 
\begin{equation*}
\begin{array}{l}
\displaystyle\sum_{\mathbf{r}}m(\mathbf{r})(Nr_1)^2\equiv\sum_{\mathbf{r}}m(\mathbf{r})(Nr_2)^2\equiv 0\pmod{\mathrm{gcd}(2,N)\cdot N},\vspace{0.2cm}\\
\displaystyle\sum_{\mathbf{r}}m(\mathbf{r})(Nr_1)(Nr_2)\equiv 0\pmod{N},\vspace{0.2cm}\\
\displaystyle\sum_{\mathbf{r}}m(\mathbf{r})\cdot\mathrm{gcd}(12,N)\equiv 0\pmod{12}.
\end{array}
\end{equation*}
Here, 
\begin{equation*}
\zeta=\prod_{\mathbf{r}}e^{\pi i r_2(1-r_1)m(\mathbf{r})}\in\mathbb{Q}(\zeta_{2N^2}).
\end{equation*}
\end{itemize}
\end{proposition}

\begin{proof}
See \cite[Proposition 2.4]{K-S}, \cite[Chapter 3, Theorem 5.2 and 5.3]{Kubert} (or \cite{Yang}).
\end{proof}

\begin{proposition}\label{Fricke and Siegel}
Let $\mathbf{r},\mathbf{s}\in\mathbb{Q}^2\setminus\mathbb{Z}^2$.
\begin{itemize}
\item[\textup{(i)}] $f_\mathbf{r}(\tau)=f_\mathbf{s}(\tau)$ if and only if $\mathbf{r}\equiv \pm\mathbf{s}\pmod{\mathbb{Z}^2}$.
\item[\textup{(ii)}] If $\mathbf{r}\not\equiv \pm\mathbf{s}\pmod{\mathbb{Z}^2}$, then we obtain
\begin{equation*}
f_\mathbf{r}(\tau)-f_\mathbf{s}(\tau)=2^{7}3^5 \frac{g_2(\tau)g_3(\tau)\eta(\tau)^4}{\Delta(\tau)}\cdot \frac{g_{\mathbf{r}+\mathbf{s}}(\tau) g_{\mathbf{r}-\mathbf{s}}(\tau)}{g_{\mathbf{r}}(\tau)^{2}g_{\mathbf{s}}(\tau)^{2}},
\end{equation*}
where
\begin{equation*}
\eta(\tau)=\sqrt{2\pi}\zeta_8 q^{1/24}\prod_{n=1}^\infty (1-q^n)
\end{equation*}
is the Dedekind $\eta$-function which is a $24$th root of $\Delta(\tau)$.
\end{itemize}
\end{proposition}

\begin{proof}
See \cite[Lemma 10.4]{Cox}, \cite[p.51]{Kubert}.

\end{proof}

\begin{remark}\label{Fricke modular unit}
Let $N~(>1)$ be an integer and $\mathbf{r},\mathbf{s}, \mathbf{r}',\mathbf{s}'\in(1/N)\mathbb{Z}^2\setminus\mathbb{Z}^2$ such that 
$\mathbf{r}\not\equiv \pm\mathbf{s}\pmod{\mathbb{Z}^2}$ and $\mathbf{r}'\not\equiv \pm\mathbf{s}'\pmod{\mathbb{Z}^2}$.
Then the function
\begin{equation*}
\frac{f_\mathbf{r}(\tau)-f_\mathbf{s}(\tau)}{f_{\mathbf{r}'}(\tau)-f_{\mathbf{s}'}(\tau)}=
\frac{g_{\mathbf{r}+\mathbf{s}}(\tau) g_{\mathbf{r}-\mathbf{s}}(\tau)}{g_{\mathbf{r}}(\tau)^{2}g_{\mathbf{s}}(\tau)^{2}}
\cdot
\frac{g_{\mathbf{r}'}(\tau)^{2}g_{\mathbf{s}'}(\tau)^{2}}{g_{\mathbf{r}'+\mathbf{s}'}(\tau) g_{\mathbf{r}'-\mathbf{s}'}(\tau)}
\end{equation*}
becomes a modular unit in $\F_N$ by Proposition \ref{property of Siegel functions} and \ref{Fricke and Siegel} (ii), which is called a \textit{Weierstrass unit} of level $N$.
\end{remark}

\section{Ring of weakly holomorphic modular functions for $\G_1(N)$}

For a positive integer $N$, let
\begin{eqnarray*}
X_1(N)&=&\G_1(N)\setminus\mathbb{H}^*\\
X^1(N)&=&\G^1(N)\setminus\mathbb{H}^*,
\end{eqnarray*}
where $\mathbb{H}^*=\mathbb{H}\cup \mathbb{Q}\cup\{i\infty\}$.
Observe that $\G^1(N)$ and $\G_1(N)$ are conjugate in $\mathrm{SL}_2(\mathbb{R})$, that is,
\begin{equation}\label{conjugate}
\o_N \G^1(N) \o_N^{-1}= \G_1(N)
\end{equation}
where $\o_N=\left[\begin{matrix}1/\sqrt{N}&0\\ 0& \sqrt{N} \end{matrix}\right]$.
Hence one can readily show that the map
\begin{equation}\label{modular curve}
\begin{array}{ccc}
X_1(N)&\longrightarrow&X^1(N)\vspace{0.1cm}\\
z&\longmapsto&\o_N^{-1}(z)=Nz
\end{array}
\end{equation}
is a well-defined isomorphism between two modular curves.
Now, let $\mathcal{F}_{1,N}(\mathbb{Q})$ (resp. $\mathcal{F}^1_{N}(\mathbb{Q})$) be the field of meromorphic modular functions for $\G_1(N)$ (resp. $\G^1(N)$) with rational Fourier coefficients.
It then follows from (\ref{conjugate}) that the map 
\begin{equation}\label{field isomorphism}
\begin{array}{ccc}
\mathcal{F}_{1,N}(\mathbb{Q})&\longrightarrow&\mathcal{F}^1_{N}(\mathbb{Q})\vspace{0.1cm}\\
h(\tau)&\longmapsto&(h\circ \o_N)(\tau)=h(\tau/N)
\end{array}
\end{equation}
is an isomorphism.
Furthermore the map (\ref{field isomorphism}) gives rise to a ring isomorphism
\begin{equation}\label{ring isomorphism}
\O_{1,N}(\mathbb{Q})\xrightarrow{~\sim~} \O^1_N(\mathbb{Q}).
\end{equation}

\begin{proposition}\label{genus zero}
The genus of $X_1(N)$ is zero if and only if $1\leq N\leq 10$ or $N=12$.
In this case, two sets of inequivalent cusps of $X_1(N)$ and $X^1(N)$ are as follows:
\begin{center}
\begin{tabular}{|c|c|c|}
\cline{1-3} 
$N$ & \textup{inequivalent cusps of $X_1(N)$}&\textup{inequivalent cusps of $X^1(N)$}\\
\cline{1-3}
$1$& $\{ i\infty  \}$ & $\{ i\infty  \}$\\
$2$& $\{0, i\infty  \}$ & $\{0, i\infty  \}$\\
$3$& $\{0, i\infty \}$ & $\{0, i\infty  \}$\\
$4$& $\{0,\frac{1}{2}, i\infty \}$& $\{0,{2}, i\infty \}$\\
$5$& $\{0, \frac{1}{2}, \frac{2}{5}, i\infty \}$ & $\{0, \frac{5}{2}, {2}, i\infty \}$\\
$6$& $\{0, \frac{1}{2}, \frac{1}{3},  i\infty \}$ & $\{0, 3, 2,  i\infty \}$\\
$7$& $\{0, \frac{1}{2}, \frac{1}{3}, \frac{2}{7}, \frac{3}{7}, i\infty  \}$ & $\{0, \frac{7}{2}, \frac{7}{3}, 2, 3, i\infty  \}$\\
$8$& $\{0, \frac{1}{2}, \frac{1}{3}, \frac{1}{4},  \frac{3}{8},  i\infty\}$& $\{0, 4, \frac{8}{3}, 2,  3,  i\infty\}$\\
$9$& $\{0, \frac{1}{2}, \frac{1}{3}, \frac{2}{3}, \frac{1}{4}, \frac{2}{9}, \frac{4}{9}, i\infty \}$ &
$\{0, \frac{9}{2}, 3, 6, \frac{9}{4}, 2, 4, i\infty \}$\\
$10$& $\{0, \frac{1}{2}, \frac{1}{3}, \frac{1}{4}, \frac{1}{5}, \frac{2}{5}, \frac{3}{10}, i\infty \}$ &
$\{0, 5, \frac{10}{3}, \frac{5}{2}, 2, 4, 3, i\infty \}$\\
$12$& $\{0, \frac{1}{2}, \frac{1}{3}, \frac{2}{3}, \frac{1}{4}, \frac{3}{4}, \frac{1}{5}, \frac{1}{6}, \frac{5}{12}, i\infty \}$ & $\{0, 6, 4, 8, 3, 9, \frac{12}{5}, 2, 5, i\infty \}$\\
\cline{1-3}
\end{tabular}
\end{center}
\end{proposition}

\begin{proof}
It is immediate from \cite{Harada} and (\ref{modular curve}).
\end{proof}
From now on, we assume that $2\leq N\leq 10$ or $N=12$.
Let $g_{1,N}(\tau)$ be the generator of the function field $\mathbb{C}(X_1(N))$ in the table (\cite[Table 2]{K-S}):\\
\begin{center}
\begin{tabular}{|c|c|}
\cline{1-2} 
$N$ & $g_{1,N}(\tau)$\\
\cline{1-2}
2& $g_{\left[\begin{smallmatrix}1/2\\ 0\end{smallmatrix}\right]}(2\tau)^{12}$\\
3& $g_{\left[\begin{smallmatrix}1/3\\ 0\end{smallmatrix}\right]}(3\tau)^{12}$\\
4& $g_{\left[\begin{smallmatrix}2/4\\ 0\end{smallmatrix}\right]}(4\tau)^{8} g_{\left[\begin{smallmatrix}1/4\\ 0\end{smallmatrix}\right]}(4\tau)^{-8}$\\
5& $g_{\left[\begin{smallmatrix}2/5\\ 0\end{smallmatrix}\right]}(5\tau)^{5} g_{\left[\begin{smallmatrix}1/5\\ 0\end{smallmatrix}\right]}(5\tau)^{-5}$\\
6& $g_{\left[\begin{smallmatrix}3/6\\ 0\end{smallmatrix}\right]}(6\tau)^{3} g_{\left[\begin{smallmatrix}1/6\\ 0\end{smallmatrix}\right]}(6\tau)^{-3}$\\
7& $g_{\left[\begin{smallmatrix}2/7\\ 0\end{smallmatrix}\right]}(7\tau)^{2}g_{\left[\begin{smallmatrix}3/7\\ 0\end{smallmatrix}\right]}(7\tau) g_{\left[\begin{smallmatrix}1/7\\ 0\end{smallmatrix}\right]}(7\tau)^{-3}$\\
8& $g_{\left[\begin{smallmatrix}3/8\\ 0\end{smallmatrix}\right]}(8\tau)^{2} g_{\left[\begin{smallmatrix}1/8\\ 0\end{smallmatrix}\right]}(8\tau)^{-2}$\\
9& $g_{\left[\begin{smallmatrix}2/9\\ 0\end{smallmatrix}\right]}(9\tau) g_{\left[\begin{smallmatrix}4/9\\ 0\end{smallmatrix}\right]}(9\tau) g_{\left[\begin{smallmatrix}1/9\\ 0\end{smallmatrix}\right]}(9\tau)^{-2}$\\
10& $g_{\left[\begin{smallmatrix}3/10\\ 0\end{smallmatrix}\right]}(10\tau) g_{\left[\begin{smallmatrix}4/10\\ 0\end{smallmatrix}\right]}(10\tau) g_{\left[\begin{smallmatrix}1/10\\ 0\end{smallmatrix}\right]}(10\tau)^{-1} g_{\left[\begin{smallmatrix}2/10\\ 0\end{smallmatrix}\right]}(10\tau)^{-1}$\\
12& $g_{\left[\begin{smallmatrix}5/12\\ 0\end{smallmatrix}\right]}(12\tau) g_{\left[\begin{smallmatrix}1/12\\ 0\end{smallmatrix}\right]}(12\tau)^{-1}$\\
\cline{1-2}
\end{tabular}
\end{center}
And, let $g^1_N(\tau)=g_{1,N}(\tau/N)$ be the function on the modular curve $X^1(N)$ induced from $g_{1,N}(\tau)$.
Observe that $g_{1,N}(\tau)$ has rational Fourier coefficients, and so it belongs to $\mathcal{F}_{1,N}(\mathbb{Q})$ (\cite[Theorem 6.7]{K-S}).

\begin{lemma}\label{field generation}
Let $N$ be as above.
Then we have
\begin{eqnarray*}
\mathcal{F}_{1,N}(\mathbb{Q})&=&\mathbb{Q}(g_{1,N}(\tau)),\\
\mathcal{F}^1_N(\mathbb{Q})&=&\mathbb{Q}(g^1_N(\tau)).
\end{eqnarray*}
\end{lemma}

\begin{proof}
If $\mathbb{C}(X_1(N))=\mathbb{C}(S)$ for some subset $S\subset \mathcal{F}_{1,N}(\mathbb{Q})$, then $\mathcal{F}_{1,N}(\mathbb{Q})=\mathbb{Q}(S)$ (\cite[Lemma 4.1]{K-S}).
Therefore, the lemma follows from (\ref{field isomorphism}).

\end{proof}
We shall first find generators of $\mathcal{O}^1_N(\mathbb{Q})$ in the above case.
Since the genus of $X^1(N)$ is zero by Proposition \ref{genus zero} and (\ref{modular curve}),  the map
\begin{equation}\label{surface isomorphism}
\begin{array}{ccc}
X^1(N)&\longrightarrow& \mathbb{P}^1(\mathbb{C})\\
\tau&\longmapsto&[g^1_N(\tau):1]
\end{array}
\end{equation}
turns out to be an isomorphism between two compact Riemann surfaces (\cite[Theorem 6.5]{K-S}).
Note that $g^1_N(i\infty)=\infty$ (\cite[Table 3]{K-S}).
Let $\tau_0$ be an inequivalent cusp of $X^1(N)$ other than $i\infty$.
Then there exists $\gamma\in\mathrm{SL}_2(\mathbb{Z})$ such that $\tau_0=\gamma(i\infty)$ and we attain
\begin{equation*}
g^1_N(\tau_0)=\lim_{\tau\rightarrow i\infty}(g^1_N\circ \gamma)(\tau).
\end{equation*}
Since the function $g^1_N(\tau)$ satisfies the assumption of Proposition \ref{property of Siegel functions} (i),
one can estimate the value $g^1_N(\tau_0)$ in a concrete way.
For example, if $N=5$ and $\tau_0=5/2$, then we derive 
\begin{eqnarray*}
g^1_5(\tau_0)=\lim_{\tau\rightarrow i\infty}(g^1_5\circ \left[\begin{smallmatrix}5&2\\ 2&1\end{smallmatrix}\right])(\tau)
&=&\lim_{\tau\rightarrow i\infty}\frac{g_{\left[\begin{smallmatrix}2\\ 4/5\end{smallmatrix}\right]}(\tau)^{5}}{g_{\left[\begin{smallmatrix}1\\ 2/5\end{smallmatrix}\right]}(\tau)^{5}}\quad \textrm{by Proposition \ref{property of Siegel functions} (i)}\\
&=&\lim_{\tau\rightarrow i\infty}\left(\frac{q^{13/12}\Z_5^2(1-q^{-1}\Z_5^{-4})(1-\Z_5^{-4})}{q^{1/12}(1-\Z_5^{-2})}\right)^5\quad\textrm{by (\ref{Siegel expansion})}\\
&=&-2-10(\Z_5+\Z_5^{-1})-5(\Z_5^2+\Z_5^{-2}).
\end{eqnarray*}
Now, let 
\begin{equation*}
C_N=\{g^1_N(\tau_0)~|~\tau_0\in\{\textrm{inequivalent cusps of $\G^1(N)$ except $i\infty$}\} \}.
\end{equation*}
Then one can readily get the following table:
\begin{center}
\begin{tabular}{|c|c|}
\cline{1-2} 
$N$ &  $C_N$\\
\cline{1-2}
2& $\{0\}$\\
\cline{1-2}
3& $\{0\}$ \\
\cline{1-2}
4& $\{16,0\}$\\
\cline{1-2}
5& $\{-2-5(\Z_5+\Z_5^{-1})-10(\Z_5^2+\Z_5^{-2}), -2-10(\Z_5+\Z_5^{-1})-5(\Z_5^2+\Z_5^{-2}), 0\}$\\
\cline{1-2}
6& $\{8, -1, 0\}$  \\
\cline{1-2}
7& $\{4+3(\Z_7+\Z_7^{-1})+\Z_7^2+\Z_7^{-2}, 4+\Z_7+\Z_7^{-1}+3(\Z_7^3+\Z_7^{-3}), 4+3(\Z_7^2+\Z_7^{-2})+\Z_7^3+\Z_7^{-3}, 1, 0   \}$\\
\cline{1-2}
8& $\{3+2(\Z_8+\Z_8^{-1}), -1, 3-2(\Z_8+\Z_8^{-1}), 1, 0 \}$\\
\cline{1-2}
9& $\{2+2(\Z_9+\Z_9^{-1})+\Z_9^2+\Z_9^{-2}, 2+\Z_9+\Z_9^{-1}+2(\Z_9^4+\Z_9^{-4}), -\Z_3, -\Z_3^2, $\\
&$2+2(\Z_9^2+\Z_9^{-2})+\Z_9^4+\Z_9^{-4},1,0   \}$ \\
\cline{1-2}
10& $\{2+\Z_{10}+\Z_{10}^{-1}+\Z_5+\Z_5^{-1}, \Z_5^{2}+\Z_5^{-2}, 2+\Z_{10}^3+\Z_{10}^{-3}+\Z_5^2+\Z_5^{-2}, \Z_5+\Z_5^{-1}, 1, -1, 0 \}$   \\
\cline{1-2}
12& $\{1+\Z_{12}+\Z_{12}^{-1}+\Z_{6}+\Z_{6}^{-1}, -1, -i, i, -\Z_3, -\Z_3^2, 1-(\Z_{12}+\Z_{12}^{-1})-(\Z_{3}+\Z_{3}^{-1}), 1, 0    \}$\\
\cline{1-2}
\end{tabular}
\end{center}
The ordering in the set $C_N$ corresponds to that of the set of inequivalent cusps of $X^1(N)$ in Proposition \ref{genus zero}.
Note that $C_N\subset \mathbb{Q}_\mathrm{ab}$ by (\ref{Siegel expansion}), where $\mathbb{Q}_\mathrm{ab}$ is the maximal abelian extension of $\mathbb{Q}$.

\begin{lemma}\label{conjugate of c}
Let $c\in C_N$ and $\sigma\in\mathrm{Gal}(\mathbb{Q}_\mathrm{ab}/\mathbb{Q})$.
Then $c^\sigma\in C_N$.
\end{lemma}
\begin{proof}
Since $\mathcal{F}_1\subset\mathcal{F}^1_N(\mathbb{Q})\subset \mathcal{F}_N$, it follow from Proposition \ref{modular function action} that $\mathcal{F}_N$ is a Galois extension of $\mathcal{F}^1_N(\mathbb{Q})$ whose Galois group is
\begin{equation}\label{Galois group}
\mathrm{Gal}(\F_N/\F^1_N(\mathbb{Q}))\cong G_N\cdot \left\{\gamma \in \mathrm{SL}_{2}(\mathbb{Z}/N\mathbb{Z})/\{\pm I_2\}~\big|~
\gamma
\equiv 
\left[\begin{matrix}
1&0\\
*&1
\end{matrix}\right]
(\bmod{N}) \right\}.
\end{equation}
Observe that $c=\displaystyle\lim_{\tau\rightarrow i\infty}(g^1_N\circ \gamma)(\tau)$
for some $\gamma\in\mathrm{SL}_2(\mathbb{Z})$, and hence we achieve by Proposition \ref{property of Siegel functions} (i) and (\ref{Siegel expansion})   
\begin{equation*}
c^\sigma=\left(\lim_{\tau\rightarrow i\infty}(g^1_N\circ \gamma)(\tau)\right)^\sigma=
\lim_{\tau\rightarrow i\infty}(g^1_N\circ \gamma)^\sigma(\tau).
\end{equation*}
Since $(g^1_N\circ \gamma)^\sigma$ is a conjugate of $g^1_N$ over $\F_1$ and $g^1_N$ is $G_N$-invariant by (\ref{Galois group}), we see from Proposition \ref{modular function action} that
\begin{equation*}
(g^1_N\circ \gamma)^\sigma=g^1_N\circ \gamma'
\end{equation*}
  for some 
$\gamma'\in\mathrm{SL}_2(\mathbb{Z})$.
Therefore, $c^\sigma\in C_N$.
\end{proof}

\begin{lemma}\label{condition for modular unit}
Let $c\in\mathbb{C}$.
Then $g^1_N(\tau)-c$ has neither zeros nor poles on $\mathbb{H}$ if and only if $c\in C_N$.
\end{lemma}

\begin{proof}
It is immediate from the isomorphism (\ref{surface isomorphism}).
\end{proof}

For each $c\in C_N$, we let $f_{N,c}(x)\in\mathbb{Q}[x]$ be the minimal polynomial of $c$ over $\mathbb{Q}$.

\begin{theorem}\label{generator}
Let $N$ be an integer such that $2\leq N\leq 10$ or $N=12$.
Then we have
\begin{eqnarray*}
\O_{1,N}(\mathbb{Q})&=&\mathbb{Q}\big[g_{1,N}(\tau), f_{N,c}(g_{1,N}(\tau))^{-1}\big]_{c\in C_N}\\
\O^1_N(\mathbb{Q})&=&\mathbb{Q}\big[g^1_N(\tau), f_{N,c}(g^1_N(\tau))^{-1}\big]_{c\in C_N}.
\end{eqnarray*}
More precisely, we attain
\begin{eqnarray*}
\O_{1,2}(\mathbb{Q})&=&\mathbb{Q}\big[g_{1,2}(\tau), g_{1,2}(\tau)^{-1}\big]\\
\O_{1,3}(\mathbb{Q})&=&\mathbb{Q}\big[g_{1,3}(\tau), g_{1,3}(\tau)^{-1}\big]\\
\O_{1,4}(\mathbb{Q})&=&\mathbb{Q}\big[g_{1,4}(\tau), g_{1,4}(\tau)^{-1}, (g_{1,4}(\tau)-16)^{-1}\big]\\
\O_{1,5}(\mathbb{Q})&=&\mathbb{Q}\big[g_{1,5}(\tau), g_{1,5}(\tau)^{-1}, (g_{1,5}(\tau)^2-11g_{1,5}(\tau)-1)^{-1}\big]\\
\O_{1,6}(\mathbb{Q})&=&\mathbb{Q}\big[g_{1,6}(\tau), g_{1,6}(\tau)^{-1}, (g_{1,6}(\tau)-8)^{-1}, (g_{1,6}(\tau)+1)^{-1}\big]\\
\O_{1,7}(\mathbb{Q})&=&\mathbb{Q}\big[g_{1,7}(\tau), g_{1,7}(\tau)^{-1}, (g_{1,7}(\tau)-1)^{-1}, (g_{1,7}(\tau)^3-8g_{1,7}(\tau)^2+5g_{1,7}(\tau)+1)^{-1}\big]\\
\O_{1,8}(\mathbb{Q})&=&\mathbb{Q}\big[g_{1,8}(\tau), g_{1,8}(\tau)^{-1}, (g_{1,8}(\tau)+1)^{-1}, (g_{1,8}(\tau)-1)^{-1}, (g_{1,8}(\tau)^2-6g_{1,8}(\tau)+1)^{-1}\big]\\
\O_{1,9}(\mathbb{Q})&=&\mathbb{Q}\big[g_{1,9}(\tau), g_{1,9}(\tau)^{-1}, (g_{1,9}(\tau)-1)^{-1}, (g_{1,9}(\tau)^2-g_{1,9}(\tau)+1)^{-1}, \\
&&\quad(g_{1,9}(\tau)^3-6g_{1,9}(\tau)^2+3g_{1,9}(\tau)+1)^{-1}\big]\\
\O_{1,10}(\mathbb{Q})&=&\mathbb{Q}\big[g_{1,10}(\tau), g_{1,10}(\tau)^{-1}, (g_{1,10}(\tau)+1)^{-1}, (g_{1,10}(\tau)-1)^{-1}, (g_{1,10}(\tau)^2+g_{1,10}(\tau)-1)^{-1}, \\
&&\quad(g_{1,10}(\tau)^2-4g_{1,10}(\tau)-1)^{-1}\big]\\
\O_{1,12}(\mathbb{Q})&=&\mathbb{Q}\big[g_{1,12}(\tau), g_{1,12}(\tau)^{-1}, (g_{1,12}(\tau)+1)^{-1}, (g_{1,12}(\tau)-1)^{-1}, (g_{1,12}(\tau)^2+1)^{-1},  \\
&&\quad (g_{1,12}(\tau)^2-g_{1,12}(\tau)+1)^{-1}, (g_{1,12}(\tau)^2-4g_{1,12}(\tau)+1)^{-1}\big].
\end{eqnarray*}
\end{theorem}

\begin{proof}
Since every zero of $f_{N,c}(x)$ lies in the set $C_N$ by Lemma \ref{conjugate of c},  $f_{N,c}(g^1_N(\tau))$ is a modular unit in $\F^1_N(\mathbb{Q})$ by Lemma \ref{condition for modular unit}.
This shows that
$\O^1_N(\mathbb{Q})\supseteq\mathbb{Q}\big[g^1_N(\tau), f_{N,c}(g^1_N(\tau))^{-1}\big]_{c\in C_N}$.
\par
Conversely, let $h(\tau)\in \O^1_N(\mathbb{Q})$.
By Lemma \ref{field generation} we can write
\begin{equation*}
h(\tau)=\frac{P(g^1_N(\tau))}{Q(g^1_N(\tau))}
\end{equation*}
for some polynomials $P(x), Q(x)\in\mathbb{Q}[x]$ which are relatively prime.
Suppose that $Q(x)$ has a zero $c_0\in\overline{\mathbb{Q}}\setminus C_N$, where $\overline{\mathbb{Q}}$ is the algebraic closure of $\mathbb{Q}$.
From Lemma \ref{condition for modular unit} we see that $g^1_N(\tau_0)-c_0=0$ for some $\tau_0\in\mathbb{H}$, which yields $Q(g^1_N(\tau_0))=0$.
Since $P(x)$ is not divisible by $x-c_0$ in $\overline{\mathbb{Q}}(x)$, we have $P(g^1_N(\tau_0))\neq 0$.
It gives a contradiction because $h(\tau)$ is weakly holomorphic.
Thus any zero of $Q(x)$ belongs to $C_N$.
Since $Q(x)\in\mathbb{Q}(x)$, we derive that $Q(x)$ has a zero $c\in C_N$ if and only if $f_{N,c}(x)~|~Q(x)$.
Therefore, we get $\O^1_N(\mathbb{Q})\subseteq\mathbb{Q}\big[g^1_N(\tau), f_{N,c}(g^1_N(\tau))^{-1}\big]_{c\in C_N}$.
And so, we have the theorem by (\ref{ring isomorphism}).

\end{proof}

\section{Generators of $\O_{1,N}(\mathbb{Q})$}
In this section, we shall construct generators of the ring $\O_{1,N}(\mathbb{Q})$ over $\mathbb{Q}$ when $N$ is divisible by $4,5,6,7$ or $9$.

\begin{lemma}\label{Fricke belongs to}
Let $m~(>1)$ be an integer.
\begin{itemize}
\item[\textup{(i)}]
$f_{\left[\begin{smallmatrix}k/m\\ 0\end{smallmatrix}\right]}(\tau)\in\O^1_m(\mathbb{Q})$ for $k\in\mathbb{Z}\setminus m\mathbb{Z}$.
\item[\textup{(ii)}] $\F^1_m(\mathbb{Q})=\F_1(f_{\left[\begin{smallmatrix}1/m\\ 0\end{smallmatrix}\right]}(\tau))$.
\end{itemize}
\end{lemma}

\begin{proof}
By Proposition \ref{properties of Fricke} we are certain that $f_{\left[\begin{smallmatrix}k/m\\ 0\end{smallmatrix}\right]}(\tau)$ is weakly holomorphic and invariant under the action of $G_m$.
Hence it has rational Fourier coefficients by Proposition \ref{modular function action}.
Furthermore, for $\gamma=\left[
\begin{matrix}
a&b\\
c&d
\end{matrix}\right] \in \G^1(m)$ we deduce by Proposition \ref{modular function action} and \ref{properties of Fricke} that
\begin{equation*}
\Big(f_{\left[\begin{smallmatrix}k/m\\ 0\end{smallmatrix}\right]}\circ\gamma\Big)(\tau)
=f_{{^t}\gamma\left[\begin{smallmatrix}k/m\\ 0\end{smallmatrix}\right]}(\tau)
=f_{\left[\begin{smallmatrix}ka/m\\ kb/m\end{smallmatrix}\right]}(\tau)
=f_{\left[\begin{smallmatrix}k/m\\ 0\end{smallmatrix}\right]}(\tau)
\end{equation*}
since $a\equiv 1\pmod{m}$ and $b\equiv 0\pmod{m}$.
Thus $f_{\left[\begin{smallmatrix}k/m\\ 0\end{smallmatrix}\right]}(\tau)$ is modular for $\G^1(m)$, which proves (i).
\par 
By (i) we have $\F_1(f_{\left[\begin{smallmatrix}1/m\\ 0\end{smallmatrix}\right]}(\tau))\subseteq\F^1_m(\mathbb{Q})$.
Let $\sigma\in\mathrm{Gal}(\F_m/\F_1(f_{\left[\begin{smallmatrix}1/m\\ 0\end{smallmatrix}\right]}(\tau)))$.
Since $f_{\left[\begin{smallmatrix}1/m\\ 0\end{smallmatrix}\right]}(\tau)$ is $G_m$-invariant, $\sigma$ is represented by $\alpha_\sigma=\left[
\begin{matrix}
a&b\\
c&d
\end{matrix}\right]\in\mathrm{SL}_2(\mathbb{Z}/m\mathbb{Z})/\{\pm I_2\}$.
Then we get
\begin{equation}\label{conjugates}
f_{\left[\begin{smallmatrix}1/m\\ 0\end{smallmatrix}\right]}(\tau)
=f_{\left[\begin{smallmatrix}1/m\\ 0\end{smallmatrix}\right]}(\tau)^\sigma
=f_{{^t}\alpha_\sigma\left[\begin{smallmatrix}1/m\\ 0\end{smallmatrix}\right]}(\tau)
=f_{\left[\begin{smallmatrix}a/m\\ b/m\end{smallmatrix}\right]}(\tau).
\end{equation}
Therefore, we achieve $a\equiv d\equiv\pm 1\pmod{m}$ and $b\equiv 0\pmod{m}$ by Proposition \ref{Fricke and Siegel} (i), which yields $\F_1(f_{\left[\begin{smallmatrix}1/m\\ 0\end{smallmatrix}\right]}(\tau))\supseteq\F^1_m(\mathbb{Q})$ by (\ref{Galois group}).
This completes the proof.

\end{proof}

For integers $m>3$ and $N>m$ such that $N\equiv 0\pmod{m}$, let
\begin{equation*}
f^1_{m,N}(\tau)=\frac{f_{\left[\begin{smallmatrix}1/N\\ 0\end{smallmatrix}\right]}(\tau)-f_{\left[\begin{smallmatrix}1/m\\ 0\end{smallmatrix}\right]}(\tau)}{f_{\left[\begin{smallmatrix}2/m\\ 0\end{smallmatrix}\right]}(\tau)-f_{\left[\begin{smallmatrix}1/m\\ 0\end{smallmatrix}\right]}(\tau)}.
\end{equation*}
Note that $f_{\left[\begin{smallmatrix}2/m\\ 0\end{smallmatrix}\right]}(\tau)\neq f_{\left[\begin{smallmatrix}1/m\\ 0\end{smallmatrix}\right]}(\tau)$ by Proposition \ref{Fricke and Siegel} (i).
It then follows from Remark \ref{Fricke modular unit} and Lemma \ref{Fricke belongs to} that $f^1_{m,N}(\tau)$ is a modular unit in $\O^1_N(\mathbb{Q})$.

\begin{theorem}\label{generator2}
Let $m>3$ and $N~(>m)$ be integers such that $N\equiv 0\pmod{m}$.
Then we have
\begin{eqnarray*}
\O_{1,N}(\mathbb{Q})&=&\O_{1,m}(\mathbb{Q})\big[f^1_{m,N}(N\tau)\big]\\
\O^1_N(\mathbb{Q})&=&\O^1_m(\mathbb{Q})\big[f^1_{m,N}(\tau)\big].
\end{eqnarray*}
\end{theorem}

\begin{proof}
It suffices to prove the last assertion by the isomorphism (\ref{ring isomorphism}).
Clearly, $\O^1_N(\mathbb{Q})\supseteq\O^1_m(\mathbb{Q})\big[f^1_{m,N}(\tau)\big]$.
\par
As for the converse inclusion, let $h(\tau)\in\O^1_N(\mathbb{Q})$.
We see from Lemma \ref{Fricke belongs to} that 
\begin{equation*}
\F^1_N(\mathbb{Q})=\F_1(f_{\left[\begin{smallmatrix}1/N\\ 0\end{smallmatrix}\right]}(\tau))
=\F^1_m(\mathbb{Q})(f^1_{m,N}(\tau)).
\end{equation*}
Thus $h=h(\tau)$ can be written as
\begin{equation}\label{expression}
h=\sum_{i=0}^{d-1}c_i f^i
\end{equation}
where $f=f^1_{m,N}(\tau)$, $d=[\F^1_N(\mathbb{Q}):\F^1_m(\mathbb{Q})]$ and $c_i\in \F^1_m(\mathbb{Q})$ for every $i$.
By multiplying both sides of (\ref{expression}) by $f^j$ for each $j=0,1,\ldots, d-1$ and taking the trace $\mathrm{Tr}=\mathrm{Tr}_{\F^1_N(\mathbb{Q})/\F^1_m(\mathbb{Q})}$, we obtain a linear system
\begin{equation*}
\left[\begin{matrix}
\mathrm{Tr}(1) &\mathrm{Tr}(f) &\cdots &\mathrm{Tr}(f^{d-1})\\ 
\mathrm{Tr}(f) &\mathrm{Tr}(f^2)&\cdots &\mathrm{Tr}(f^d)\\
\vdots&\vdots&&\vdots\\
\mathrm{Tr}(f^{d-1}) &\mathrm{Tr}(f^d) &\cdots& \mathrm{Tr}(f^{2d-2})
\end{matrix}\right]
\left[\begin{matrix}
c_0\\ 
c_1\\
\vdots\\
c_{d-1}
\end{matrix}\right]
=\left[\begin{matrix}
\mathrm{Tr}(h)\\ 
\mathrm{Tr}(fh)\\
\vdots\\
\mathrm{Tr}(f^{d-1}h)
\end{matrix}\right].
\end{equation*}
For simplicity,  put
\begin{equation*}
T=\left[\begin{matrix}
\mathrm{Tr}(1) &\mathrm{Tr}(f) &\cdots &\mathrm{Tr}(f^{d-1})\\ 
\mathrm{Tr}(f) &\mathrm{Tr}(f^2)&\cdots &\mathrm{Tr}(f^d)\\
\vdots&\vdots&&\vdots\\
\mathrm{Tr}(f^{d-1}) &\mathrm{Tr}(f^d) &\cdots& \mathrm{Tr}(f^{2d-2})
\end{matrix}\right].
\end{equation*}
Since $f,h\in\O^1_N(\mathbb{Q})$, we attain $c_i\in\det(T)^{-1}\O^1_m(\mathbb{Q})$ for $i=0,1,\ldots, d-1$.
Now, let $f_1,f_2,\ldots,f_d$ be the Galois conjugates of $f$ over $\F^1_m(\mathbb{Q})$.
Then, by the Vandermonde determinant formula, we get
\begin{eqnarray*}
\det(T)&=&\left|\begin{matrix}
\sum_{k=1}^d f_k^0 &\sum_{k=1}^d f_k^1 &\cdots&\sum_{k=1}^d f_k^{d-1}\\ 
\sum_{k=1}^d f_k^1 &\sum_{k=1}^d f_k^2 &\cdots&\sum_{k=1}^d f_k^{d}\\
\vdots&\vdots&&\vdots\\
\sum_{k=1}^d f_k^{d-1} &\sum_{k=1}^d f_k^d &\cdots&\sum_{k=1}^d f_k^{2d-2}
\end{matrix}\right|\\
&=&\left|\begin{matrix}
f_1^0 & f_2^0 &\cdots& f_d^0\\ 
f_1^1 & f_2^1 &\cdots& f_d^1\\ 
\vdots&\vdots&&\vdots\\
f_1^{d-1} & f_2^{d-1} &\cdots& f_d^{d-1}
\end{matrix}\right|\cdot
\left|\begin{matrix}
f_1^0 & f_1^1 &\cdots& f_1^{d-1}\\ 
f_2^0 & f_2^1 &\cdots& f_2^{d-1}\\ 
\vdots&\vdots&&\vdots\\
f_d^0 & f_d^1 &\cdots& f_d^{d-1}
\end{matrix}\right|\\
&=&\prod_{1\leq i<j\leq d}(f_i-f_j)^2.
\end{eqnarray*}
It follows from Proposition \ref{Fricke and Siegel}, Lemma \ref{Fricke belongs to} and (\ref{conjugates}) that each $f_i-f_j$ is of the form 
\begin{equation*}
\frac{f_{\left[\begin{smallmatrix}a/N\\ b/N\end{smallmatrix}\right]}(\tau)-f_{\left[\begin{smallmatrix}1/m\\ 0\end{smallmatrix}\right]}(\tau)}{f_{\left[\begin{smallmatrix}2/m\\ 0\end{smallmatrix}\right]}(\tau)-f_{\left[\begin{smallmatrix}1/m\\ 0\end{smallmatrix}\right]}(\tau)}
-\frac{f_{\left[\begin{smallmatrix}c/N\\ d/N\end{smallmatrix}\right]}(\tau)-f_{\left[\begin{smallmatrix}1/m\\ 0\end{smallmatrix}\right]}(\tau)}{f_{\left[\begin{smallmatrix}2/m\\ 0\end{smallmatrix}\right]}(\tau)-f_{\left[\begin{smallmatrix}1/m\\ 0\end{smallmatrix}\right]}(\tau)}
=\frac{f_{\left[\begin{smallmatrix}a/N\\ b/N\end{smallmatrix}\right]}(\tau)-f_{\left[\begin{smallmatrix}c/N\\ d/N\end{smallmatrix}\right]}(\tau)}{f_{\left[\begin{smallmatrix}2/m\\ 0\end{smallmatrix}\right]}(\tau)-f_{\left[\begin{smallmatrix}1/m\\ 0\end{smallmatrix}\right]}(\tau)}
\end{equation*}
for some $a,b,c,d\in\mathbb{Z}$ such that $\left[\begin{matrix}a/N\\ b/N\end{matrix}\right]\not\equiv\pm\left[\begin{matrix}c/N\\ d/N\end{matrix}\right]\pmod{\mathbb{Z}^2}$ and $\gcd(a,b,N)= \gcd(c,d,N)=1$.
Hence $\det(T)$ is a modular unit in $\O^1_m(\mathbb{Q})$ by Remark \ref{Fricke modular unit}, and so $c_i\in\O^1_m(\mathbb{Q})$ for every $i$.
Therefore, $h(\tau)\in\O^1_m(\mathbb{Q})\big[f^1_{m,N}(\tau)\big]$ as desired.
\end{proof}

\begin{corollary}\label{main corollary}
Let $N$ be a positive integer.
\begin{itemize}
\item[\textup{(i)}]
If $N>4$ and $N\equiv 0\pmod{4}$, then we have
\begin{eqnarray*}
\O_{1,N}(\mathbb{Q})&=&\mathbb{Q}\big[g_{1,4}(\tau), g_{1,4}(\tau)^{-1}, (g_{1,4}(\tau)-16)^{-1}, f^1_{4,N}(N\tau)\big]\\
\O^1_N(\mathbb{Q})&=&\mathbb{Q}\big[g^1_4(\tau), g^1_4(\tau)^{-1}, (g^1_4(\tau)-16)^{-1}, f^1_{4,N}(\tau)\big].
\end{eqnarray*}
\item[\textup{(ii)}]
If $N>5$ and $N\equiv 0\pmod{5}$, then we have
\begin{eqnarray*}
\O_{1,N}(\mathbb{Q})&=&\mathbb{Q}\big[g_{1,5}(\tau), g_{1,5}(\tau)^{-1}, (g_{1,5}(\tau)^2-11g_{1,5}(\tau)-1)^{-1}, f^1_{5,N}(N\tau)\big]\\
\O^1_N(\mathbb{Q})&=&\mathbb{Q}\big[g^1_5(\tau), g^1_5(\tau)^{-1}, (g^1_5(\tau)^2-11g^1_5(\tau)-1)^{-1}, f^1_{5,N}(\tau)\big].
\end{eqnarray*}
\item[\textup{(iii)}]
If $N>6$ and $N\equiv 0\pmod{6}$, then we have
\begin{eqnarray*}
\O_{1,N}(\mathbb{Q})&=&\mathbb{Q}\big[g_{1,6}(\tau), g_{1,6}(\tau)^{-1}, (g_{1,6}(\tau)-8)^{-1}, (g_{1,6}(\tau)+1)^{-1}, f^1_{6,N}(N\tau)\big]\\
\O^1_N(\mathbb{Q})&=&\mathbb{Q}\big[g^1_6(\tau), g^1_6(\tau)^{-1}, (g^1_6(\tau)-8)^{-1}, (g^1_6(\tau)+1)^{-1}, f^1_{6,N}(\tau)\big].
\end{eqnarray*}
\item[\textup{(iv)}]
If $N>7$ and $N\equiv 0\pmod{7}$, then we have
\begin{eqnarray*}
\O_{1,N}(\mathbb{Q})&=&\mathbb{Q}\big[g_{1,7}(\tau), g_{1,7}(\tau)^{-1}, (g_{1,7}(\tau)-1)^{-1}, (g_{1,7}(\tau)^3-8g_{1,7}(\tau)^2+5g_{1,7}(\tau)+1)^{-1},\\
&&\quad f^1_{7,N}(N\tau)\big]\\
\O^1_N(\mathbb{Q})&=&\mathbb{Q}\big[g^1_7(\tau), g^1_7(\tau)^{-1}, (g^1_7(\tau)-1)^{-1}, (g^1_7(\tau)^3-8g^1_7(\tau)^2+5g^1_7(\tau)+1)^{-1}, f^1_{7,N}(\tau)\big].
\end{eqnarray*}
\item[\textup{(v)}]
If $N>9$ and $N\equiv 0\pmod{9}$, then we have
\begin{eqnarray*}
\O_{1,N}(\mathbb{Q})&=&\mathbb{Q}\big[g_{1,9}(\tau), g_{1,9}(\tau)^{-1}, (g_{1,9}(\tau)-1)^{-1}, (g_{1,9}(\tau)^2-g_{1,9}(\tau)+1)^{-1}, \\
&&\quad(g_{1,9}(\tau)^3-6g_{1,9}(\tau)^2+3g_{1,9}(\tau)+1)^{-1}, f^1_{9,N}(N\tau)\big]\\
\O^1_N(\mathbb{Q})&=&\mathbb{Q}\big[g^1_9(\tau), g^1_9(\tau)^{-1}, (g^1_9(\tau)-1)^{-1}, (g^1_9(\tau)^2-g^1_9(\tau)+1)^{-1}, \\
&&\quad(g^1_9(\tau)^3-6g^1_9(\tau)^2+3g^1_9(\tau)+1)^{-1}, f^1_{9,N}(\tau)\big].
\end{eqnarray*}

\end{itemize}
\end{corollary}

\begin{proof}
It is immediate from Theorem \ref{generator} and \ref{generator2}.
\end{proof}

\begin{remark}
In order to construct $\O_{1,N}(\mathbb{Q})$ for arbitrary $N$, it suffices to find generators of the ring $\O_{1,p}(\mathbb{Q})$ for an odd prime $p\geq 11$.
\end{remark}

\section{Fricke families of level $N$}\label{Fricke families of level $N$}

For an integer $N~(>1)$, let
\begin{equation*}
\mathcal{R}_N=\{\mathbf{r}\in\mathbb{Q}^2~|~\textrm{$\mathbf{r}$ has the primitive denominator $N$} \}.
\end{equation*}
A family $\{h_\mathbf{r}(\tau)\}_{\mathbf{r}\in\mathcal{R}_N}$ of functions in $\F_N$ is called a \textit{Fricke family} of level $N$ if it satisfies the following conditions:
\begin{itemize}
\item[\textup{(i)}] Each $h_\mathbf{r}(\tau)$ is weakly holomorphic.
\item[\textup{(ii)}] $h_\mathbf{r}(\tau)$ depends only on $\pm\mathbf{r}\pmod{\mathbb{Z}^2}$.
\item[\textup{(iii)}] $h_\mathbf{r}(\tau)^\alpha=h_{{^t}\alpha\mathbf{r}}(\tau)$ for all $\alpha\in\mathrm{GL}_2(\mathbb{Z}/N\mathbb{Z})/\{\pm I_2\}$. 
\end{itemize}
Note that the set of Fricke functions $\{f_\mathbf{r}(\tau)\}_{\mathbf{r}\in\mathcal{R}_N}$ is a Fricke family of level $N$ by Proposition \ref{properties of Fricke}.
In this section, by $\mathrm{Fr}(N)$ we mean the set of all Fricke families of level $N$.
Then one can endow this set the natural ring structure:
\begin{eqnarray*}
\{h_\mathbf{r}(\tau)\}_{\mathbf{r}\in\mathcal{R}_N}+\{k_\mathbf{r}(\tau)\}_{\mathbf{r}\in\mathcal{R}_N}&=&\{(h_\mathbf{r}+k_\mathbf{r})(\tau)\}_{\mathbf{r}\in\mathcal{R}_N}\\
\{h_\mathbf{r}(\tau)\}_{\mathbf{r}\in\mathcal{R}_N}\cdot\{k_\mathbf{r}(\tau)\}_{\mathbf{r}\in\mathcal{R}_N}&=&\{(h_\mathbf{r}k_\mathbf{r})(\tau)\}_{\mathbf{r}\in\mathcal{R}_N}.
\end{eqnarray*}
Eum et al (\cite{E-K-S}) classify all Fricke families of level $N$ when $N\equiv 0\pmod{4}$ by using the relation between $\mathrm{Fr}(N)$ and $\O^1_N(\mathbb{Q})$ as follows:

\begin{proposition}\label{Fricke family}
For an integer $N>1$, the map
\begin{equation*}
\begin{array}{ccc}
\mathrm{Fr}(N) &\longrightarrow & \O^1_N(\mathbb{Q})\\
\{h_\mathbf{r}(\tau)\}_{\mathbf{r}\in\mathcal{R}_N}&\longmapsto& h_{\left[\begin{smallmatrix}1/N\\ 0\end{smallmatrix}\right]}(\tau)
\end{array}
\end{equation*}
is a well-defined ring isomorphism.
\end{proposition}
\begin{proof}
\cite[Theorem 4.3]{E-K-S}.
\end{proof}

\begin{theorem}\label{classify Fricke families}
Let $N>1$ be a positive integer.
Then a family $\{h_\mathbf{r}(\tau)\}_{\mathbf{r}\in\mathcal{R}_N}$ of functions in $\F_N$ is a Fricke family of level $N$ if and only if the following condition holds:
\begin{itemize}
\item[\textup{(i)}] 
When $N\equiv 0\pmod{4}$, there exists a polynomial $P(x_1,x_2,x_3,x_4)\in\mathbb{Q}[x_1,x_2,x_3,x_4]$ such that
\begin{equation*}
h_\mathbf{r}(\tau)=P\big(\mathfrak{g}_\mathbf{r}(\tau), \mathfrak{g}_\mathbf{r}(\tau)^{-1}, (\mathfrak{g}_\mathbf{r}(\tau)-16)^{-1}, \mathfrak{f}_\mathbf{r}(\tau)\big)
\end{equation*}
for all $\mathbf{r}\in\mathcal{R}_N$, where 
\begin{equation*}
\mathfrak{g}_\mathbf{r}(\tau)=\frac{g_{(N/2)\mathbf{r}}(\tau)^{8}}{g_{(N/4)\mathbf{r}}(\tau)^{8}}\quad\textrm{and}\quad
\mathfrak{f}_\mathbf{r}(\tau)=\frac{f_{\mathbf{r}}(\tau)-f_{(N/4)\mathbf{r}}(\tau)}{f_{(N/2)\mathbf{r}}(\tau)-f_{(N/4)\mathbf{r}}(\tau)}.
\end{equation*}

\item[\textup{(ii)}]
When $N\equiv 0\pmod{5}$, there exists a polynomial $P(x_1,x_2,x_3,x_4)\in\mathbb{Q}[x_1,x_2,x_3,x_4]$ such that
\begin{equation*}
h_\mathbf{r}(\tau)=P\big(\mathfrak{g}_\mathbf{r}(\tau), \mathfrak{g}_\mathbf{r}(\tau)^{-1}, (\mathfrak{g}_\mathbf{r}(\tau)^2-11\mathfrak{g}_\mathbf{r}(\tau)-1)^{-1}, \mathfrak{f}_\mathbf{r}(\tau)\big)
\end{equation*}
for all $\mathbf{r}\in\mathcal{R}_N$, where 
\begin{equation*}
\mathfrak{g}_\mathbf{r}(\tau)=\frac{g_{(2N/5)\mathbf{r}}(\tau)^{5}}{g_{(N/5)\mathbf{r}}(\tau)^{5}}\quad\textrm{and}\quad
\mathfrak{f}_\mathbf{r}(\tau)=\frac{f_{\mathbf{r}}(\tau)-f_{(N/5)\mathbf{r}}(\tau)}{f_{(2N/5)\mathbf{r}}(\tau)-f_{(N/5)\mathbf{r}}(\tau)}.
\end{equation*}

\item[\textup{(iii)}]
When $N\equiv 0\pmod{6}$, there exists a polynomial $P(x_1,x_2,x_3,x_4,x_5)\in\mathbb{Q}[x_1,x_2,x_3,x_4,x_5]$ such that
\begin{equation*}
h_\mathbf{r}(\tau)=P\big(\mathfrak{g}_\mathbf{r}(\tau), \mathfrak{g}_\mathbf{r}(\tau)^{-1}, (\mathfrak{g}_\mathbf{r}(\tau)-8)^{-1}, (\mathfrak{g}_\mathbf{r}(\tau)+1)^{-1}, \mathfrak{f}_\mathbf{r}(\tau)\big)
\end{equation*}
for all $\mathbf{r}\in\mathcal{R}_N$, where 
\begin{equation*}
\mathfrak{g}_\mathbf{r}(\tau)=\frac{g_{(N/2)\mathbf{r}}(\tau)^{3}}{g_{(N/6)\mathbf{r}}(\tau)^{3}}\quad\textrm{and}\quad
\mathfrak{f}_\mathbf{r}(\tau)=\frac{f_{\mathbf{r}}(\tau)-f_{(N/6)\mathbf{r}}(\tau)}{f_{(N/3)\mathbf{r}}(\tau)-f_{(N/6)\mathbf{r}}(\tau)}.
\end{equation*}

\item[\textup{(iv)}]
When $N\equiv 0\pmod{7}$, there exists a polynomial $P(x_1,x_2,x_3,x_4,x_5)\in\mathbb{Q}[x_1,x_2,x_3,x_4,x_5]$ such that
\begin{equation*}
h_\mathbf{r}(\tau)=P\big(\mathfrak{g}_\mathbf{r}(\tau), \mathfrak{g}_\mathbf{r}(\tau)^{-1}, 
(\mathfrak{g}_\mathbf{r}(\tau)-1)^{-1}, 
(\mathfrak{g}_\mathbf{r}(\tau)^3-8\mathfrak{g}_\mathbf{r}(\tau)^2+5\mathfrak{g}_\mathbf{r}(\tau)+1)^{-1}, \mathfrak{f}_\mathbf{r}(\tau)\big)
\end{equation*}
for all $\mathbf{r}\in\mathcal{R}_N$, where 
\begin{equation*}
\mathfrak{g}_\mathbf{r}(\tau)=\frac{g_{(2N/7)\mathbf{r}}(\tau)^{2}g_{(3N/7)\mathbf{r}}(\tau)}{g_{(N/7)\mathbf{r}}(\tau)^{3}}\quad\textrm{and}\quad
\mathfrak{f}_\mathbf{r}(\tau)=\frac{f_{\mathbf{r}}(\tau)-f_{(N/7)\mathbf{r}}(\tau)}{f_{(2N/7)\mathbf{r}}(\tau)-f_{(N/7)\mathbf{r}}(\tau)}.
\end{equation*}

\item[\textup{(v)}]
When $N\equiv 0\pmod{9}$, there exists a polynomial $P(x_1,x_2,x_3,x_4,x_5,x_6)\in\mathbb{Q}[x_1,x_2,x_3,x_4,x_5,x_6]$ such that
\begin{equation*}
\begin{array}{ccl}
h_\mathbf{r}(\tau)&=&P\big(\mathfrak{g}_\mathbf{r}(\tau), \mathfrak{g}_\mathbf{r}(\tau)^{-1}, 
(\mathfrak{g}_\mathbf{r}(\tau)-1)^{-1}, 
(\mathfrak{g}_\mathbf{r}(\tau)^2-\mathfrak{g}_\mathbf{r}(\tau)+1)^{-1},\\
&&\quad
(\mathfrak{g}_\mathbf{r}(\tau)^3-6\mathfrak{g}_\mathbf{r}(\tau)^2+3\mathfrak{g}_\mathbf{r}(\tau)+1)^{-1}, \mathfrak{f}_\mathbf{r}(\tau)\big)
\end{array}
\end{equation*}
for all $\mathbf{r}\in\mathcal{R}_N$, where 
\begin{equation*}
\mathfrak{g}_\mathbf{r}(\tau)=\frac{g_{(2N/9)\mathbf{r}}(\tau)g_{(4N/9)\mathbf{r}}(\tau)}{g_{(N/9)\mathbf{r}}(\tau)^{2}}\quad\textrm{and}\quad
\mathfrak{f}_\mathbf{r}(\tau)=\frac{f_{\mathbf{r}}(\tau)-f_{(N/9)\mathbf{r}}(\tau)}{f_{(2N/9)\mathbf{r}}(\tau)-f_{(N/9)\mathbf{r}}(\tau)}.
\end{equation*}
\end{itemize}
\end{theorem}

\begin{proof}
For $\mathbf{r}=\left[\begin{matrix}a/N\\ b/N\end{matrix}\right]\in\mathcal{R}_N$, we can take $\alpha_\mathbf{r}=\left[\begin{matrix}a&b\\ c&d\end{matrix}\right]\in\mathrm{SL}_2(\mathbb{Z}/N\mathbb{Z})$ with $c,d\in\mathbb{Z}$.
For each $m=4,5,6,7$ and $9$, if $N>m$ and $N\equiv 0\pmod{m}$, then we achieve by Proposition \ref{property of Siegel functions} (i) that 
\begin{eqnarray*}
g^1_m(\tau)^{\alpha_\mathbf{r}}&=&\mathfrak{g}_\mathbf{r}(\tau)\\
f^1_{m,N}(\tau)^{\alpha_\mathbf{r}}&=&\mathfrak{f}_\mathbf{r}(\tau).
\end{eqnarray*}
Here we note that $\mathfrak{f}_\mathbf{r}(\tau)=0$ if $N=m$.
Then the theorem is an immediate consequence of Theorem \ref{generator}, Corollary \ref{main corollary} and Proposition \ref{Fricke family}.

\end{proof}

\bibliographystyle{amsplain}

\address{% First Author
Ja Kyung Koo\\
Department of Mathematical Sciences \\
KAIST \\
Daejeon 305-701 \\
Republic of Korea} {jkkoo@math.kaist.ac.kr}
%%%%%%%%%
\address{% Corresponding Author
Dong Sung Yoon\\
National Institute for Mathematical Sciences \\
Daejeon 305-811 \\
Republic of Korea} {dsyoon@nims.re.kr}

\end{document}